\documentclass[10pt,oneside]{amsart} 

\usepackage{amsmath,amstext,amscd, amssymb, color, array} 
\usepackage{pinlabel}
\usepackage[normalem]{ulem}
\usepackage[colorlinks=true, pdfstartview=FitV, linkcolor=blue, citecolor=blue, urlcolor=blue]{hyperref}
 \usepackage[abs]{overpic}
\usepackage{xcolor,varwidth}
\usepackage{enumerate, scalerel}

\makeatletter
\def\@tocline#1#2#3#4#5#6#7{\relax
\ifnum #1>\c@tocdepth 
  \else 
    \par \addpenalty\@secpenalty\addvspace{#2}%
\begingroup \hyphenpenalty\@M
    \@ifempty{#4}{%
      \@tempdima\csname r@tocindent\number#1\endcsname\relax
 }{%
   \@tempdima#4\relax
 }%
 \parindent\z@ \leftskip#3\relax \advance\leftskip\@tempdima\relax
 \rightskip\@pnumwidth plus4em \parfillskip-\@pnumwidth
 #5\leavevmode\hskip-\@tempdima #6\nobreak\relax
 \ifnum#1<0\hfill\else\dotfill\fi\hbox to\@pnumwidth{\@tocpagenum{#7}}\par
 \nobreak
 \endgroup
  \fi}
\makeatother
\let\oldtocsection=\tocsection
\let\oldtocsubsection=\tocsubsection
\let\oldtocsubsubsection=\tocsubsubsection
\renewcommand{\tocsection}[2]{\hspace{0em}\oldtocsection{#1}{#2}}
\renewcommand{\tocsubsection}[2]{\hspace{1em}\oldtocsubsection{#1}{#2}}
\renewcommand{\tocsubsubsection}[2]{\hspace{2em}\oldtocsubsubsection{#1}{#2}}

\definecolor{cerulean}{rgb}{0,.48,.65} 
\definecolor{magenta}{rgb}{.5,0,.5} 
\definecolor{dred}{rgb}{.5,0,0} 
\definecolor{green}{rgb}{0,.5,0} 
\definecolor{blue}{rgb}{0,0,1} 
\definecolor{black}{rgb}{0,0,0} 
\definecolor{dgreen}{rgb}{0,.3,0} 
\definecolor{vdred}{rgb}{.3,0,0} 
\definecolor{red}{rgb}{1,0,0} 
\definecolor{salmon}{rgb}{0.98,0.50,0.45} 
\definecolor{gray}{rgb}{.5,.5,.5} 
\definecolor{seagreen}{rgb}{0.13,0.70,0.67} 
\definecolor{chartreuse}{rgb}{0.40,0.80,0.00}
\definecolor{cornflower}{rgb}{0.39,0.58,0.93} 

\definecolor{gold}{rgb}{0.65,0.45,0.00}
 
\theoremstyle{plain}

\newtheorem{thm}{Theorem}

\newtheorem*{conjproblem*}{The conjugacy problem}
\newtheorem*{0twistedproblem*}{The 0-twisted-conjugacy problem}
\newtheorem*{Htwistedproblem*}{The H-twisted conjugacy problem}
\newtheorem*{Itwistedproblem*}{The I-twisted conjugacy problem}

\theoremstyle{definition}

\newtheorem{Open questions}[thm]{Open questions}
\newtheorem{Open question}[thm]{Open question}
\newtheorem{Open problems}[thm]{Open problems}
\newtheorem{Open problem}[thm]{Open problem}

\def\Bbb{\mathbb}

\def\Z{\Bbb{Z}}
\def\N{\Bbb{N}}

\def\Q{\Bbb{Q}}

\def\ni{\noindent}

\def\CL{\hbox{\rm CL}}

\def\F+L{\hbox{$\textup{F}\!_+\textup{L}$}}

\def\SL{\hbox{\rm SL}}

\def\onto{{\kern3pt\to\kern-8pt\to\kern3pt}}

\def\<{\langle}
\def\>{\rangle}
\def\|{{\ |\ }}

\def\G{\Gamma}

\newcommand{\set}[1]{\left\{#1\right\}}

\renewcommand{\ni}{\noindent}

\def\*{^{\star}}

\def\l{\lambda}

\begin{document}

\title[Conjugators in 2-step nilpotent groups]{Linear Diophantine equations and conjugator length in 2-step nilpotent groups}

\author{M.\ R.\ Bridson and T.\ R.\ Riley}

\date{31 May 2025; revised 7 January 2026}

\begin{abstract}
\ni  We establish upper bounds on the lengths of minimal conjugators in 2-step nilpotent groups.  These
bounds exploit the existence of small integral solutions to systems of linear Diophantine equations.  
We prove that in some cases these bounds are sharp. This enables us to construct a
family of finitely generated 2-step nilpotent groups $(G_m)_{m\in\mathbb{N}}$ such that 
the conjugator length function of $G_m$ grows like a polynomial of degree $m+1$.

  \smallskip
\ni \footnotesize{\textbf{2020 Mathematics Subject Classification:  20F65, 20F10, 20F18}}  \\ 
\ni \footnotesize{\emph{Key words and phrases:} nilpotent groups,  conjugator length, linear Diophantine equations}
\end{abstract}

\thanks{The first author thanks the Mathematics Department of Stanford University for its hospitality.
The second author gratefully acknowledges the financial support of the National Science Foundation (NSF GCR-2428489).  ORCID: 0000-0002-0080-9059 (MRB),  0009-0004-3699-0322 (TRR)}

\maketitle

\section{Introduction}
In this article we will explore the difficulty of the conjugacy problem in 2-step nilpotent groups, using normal forms
to convert each instance of the problem into a system of linear Diophantine equations.
A natural measure of the
difficulty of the conjugacy problem in a finitely generated
group $G$ is its \emph{conjugator length function}  
$$\CL(n) \ = \ \max  \set{ \CL(u,v) \mid \text{ words }  u  \text{ and } v \text{  with } u \! \sim v \text{ in } G  \text{ and } |u| + |v| \leq n  },$$
where $u \sim v$ denotes conjugacy in $G$ and $ \CL(u,v)$ is defined to be the length $|w|$ of a shortest word $w$  such that $uw=wv$ in $G$.   The precise values of the function $\CL(n)$ depend on the choice of finite generating set,  but the
$\simeq$ class of $\CL(n)$ does not, where  $\simeq$ is the following standard equivalence relation on functions
$f,g:\mathbb{N} \to \mathbb{N}$: by definition,  $f \preceq g$ if there exists $C>0$ such that   $f(n) \leq Cg(  Cn+C ) +C$  for all $n \in \N$; if $f \preceq g$ and $g \preceq f$ then $f\simeq g$.  If $\CL(n)\simeq n^d$, then one says that
$\CL(n)$ is {\em polynomial of degree $d$}.

Here and in a companion article \cite{BrRi2} we describe the first families of groups to exhibit the following behaviour.

\begin{thm} \label{polynomial examples thm}
For all integers $d \geq 1$ there   is a finitely presented group with conjugator length function $\CL(n) \simeq n^d$. 	
\end{thm}

The groups  $\G_d$ we use to prove this theorem
 in \cite{BrRi2} are standard lattices in the much-studied model filiform groups.  The proof that  $\CL_{\G_d}(n) \simeq n^d$ proceeds by induction on $d$; it relies on the fact that $\G_{d}$ is isomorphic to $\G_{d+1}$
modulo its  cyclic centre, and hence on the fact that $\G_d$ is nilpotent of class $d$.  The proof involves a careful
analysis of the geometry of cyclic subgroups and centralisers in $\G_d$.  

The groups $G_m$ that we
will construct here to prove Theorem \ref{polynomial examples thm} are of a quite different nature.
First of all, they are not drawn from a well known family of prototypes: they are bespoke, 
designed for the sole purpose
of ensuring  that their conjugator length functions are polynomial of arbitrary degree.  
The construction of $G_m$
is not overtly geometric and the tools that we use to study these groups involve no geometry.
The  crucial property
of $G_m$ is that the Diophantine equations that arise from comparing normal forms for certain elements in 
the group have a simple recursive structure that enables us to establish a lower bound on $\CL_{G_m}(n) $. 
In a subsequent article we shall explain how this property can be exploited so as to construct amalgamated
free products with more exotic conjugator length functions. 

A further noteworthy feature 
of the  groups  $G_m$ is that they are all {\em nilpotent of class $2$}.  Thus they reveal
a sharp difference in behaviour between
the word problem and the conjugacy problem: for a  finitely generated nilpotent group $G$ of class $c$, the Dehn function,
which measures the complexity of the word problem,   is polynomial of degree at most $c+1$ (see \cite{GHR, Gromov6}); 
but we now know that the  conjugator length function of $G$ can be polynomial of any degree,  even when $c=2$. 

Our groups $G_m$ are central extensions 
$$
1\to \Z^m\to G_m \to \Z^{m+2}\to 1
$$
of $\Z^{m+2} = \langle a_1, \ldots, a_m, b_1, b_2\rangle$ by $\Z^m  = \langle c_1, \ldots, c_m \rangle$.

\begin{thm} \label{central extensions CL lower bound examples}
For $m \geq 1$, the group $$G_m \ := \ \scaleleftright[1.75ex]{\bigg\langle}{    \  \parbox{17mm}{ $a_1, \ldots, a_m$ \\ $b_1, b_2$ \\ $c_1, \ldots, c_m$} \rule{0mm}{12mm}  \left|  \ \rule{0mm}{12mm} \parbox{73mm}{$a_i a_j = a_j a_i$ for all $i,j$   \\ \ $b_1 b_2 =b_2 b_1$ \\  $ b_1 a_i = a_i b_1 c_i$ \  for \ $i=1, \ldots, m$  \\ $b_2 a_i = a_i b_2 c_{i+1}^{-1}$  \  for \ $i=1, \ldots, m-1$  \\ $b_2 a_m  = a_m b_2$ \\ $c_ic_j = c_jc_i$,  \ $a_ic_j = c_j a_i$,   \ $b_ic_j = c_j b_i$ for all  $i,j$   }  \right. }{\bigg\rangle}$$
has $\CL(n) \simeq n^{m+1}$.   
\end{thm}

The upper bound in Theorem~\ref{central extensions CL lower bound examples} is a special case of the
following general result.

\begin{thm} \label{thm: class-2 upper bound}
The conjugator length functions of finitely generated class-2 nilpotent groups can be bounded from above as follows.
Let  $G$ be a finitely generated group  that is a central extension
$$ 1\to Z \to G \to A\to 1$$
where $A$ is abelian and
$$Z \  \cong \  \Z^m \times T$$
where $T$ is a finite abelian group.  Then $G$ has $\CL(n) \preceq n^{m + 1}$.  
\end{thm}

The context of this theorem among prior literature on conjugator length functions of nilpotent groups is as follows. Ji, Ogle, and Ramsey \cite{JOR}   argued that the conjugator length functions  of finitely generated class-$2$  nilpotent groups  grow at
 most polynomially. Macdonald,   Myasnikov,   Nikolaev,  and Vassileva \cite[Thm.\ 4.7]{MMNV} proved polynomial  upper bounds for the conjugator length functions of all finitely generated nilpotent groups.
Their argument  gives an upper bound of $2^m(6mc^2)^{m^2}$ on the  degree, where $c$ is the   class and $m$ is the number of elements in what they call a \emph{Mal'cev basis} for the group.  In the case of a class-2 nilpotent group $G$ with  
torsion-free center $Z$,  the union of a basis for $Z$ and a set of elements of $G$ that map to a basis for $G/Z$ is a Mal'cev basis.    
 
The tighter bounds that we get in the class-$2$ case compared to \cite{MMNV} will be obtained
by carefully reducing each search for conjugators (where they are known to exist) to 
a consistent system of linear Diophantine equations, and then applying a result  of Borosh, Flahive, Rubin, and Treybig \cite{BFRT-Diophantine}  that bounds the size of the smallest integral solution.

For more background and information about 
conjugator length functions,  we refer the reader to our survey article with Andrew Sale \cite{BrRiSa}.  In that article
we proved that the conjugator length function of the 3-dimensional integral Heisenberg group $\mathcal{H}_3(\Z)$ is quadratic.  
The techniques that we deploy here are inspired in part by the proof of that result. 
In fact, the group $G_1$ in Theorem~\ref{central extensions CL lower bound examples}   is $\mathcal{H}_3(\Z) \times \Z$ with the $\Z$-factor being generated by $b_2$.   
 
\subsection*{Acknowledgement}  We thank the referees for their helpful comments.

\section{Conjugacy in class-2 nilpotent groups}

Our aim here is to prove Theorem~\ref{thm: class-2 upper bound}.  
We  have a  finitely generated group  $G$  that is a central extension
$$1\to Z \to G \to A\to 1$$
with $A$ abelian and
$$Z \  \cong \  \Z^m \times C_{o_1}  \times \cdots \times C_{o_l},$$
where each $C_{o_i}$ is a cyclic group of finite order $o_i$.  Let $r =m+l$. 

We express $A$ as a direct sum of cyclic groups and choose preimages
$a_1, \ldots, a_k \in G$  for   generators of these cyclic groups.  
We also fix  $c_1, \ldots, c_{r}\in Z$ so that $c_1, \ldots, c_{m}$  generate the $\Z^m$ summand 
and   
$c_{m+j}$, for $j= 1, \ldots, l$, generates the summand $C_{o_j}$.  This gives us a set of generators 
$S= \set{a_1, \ldots, a_k, c_1, \ldots, c_r}$ for $G$.  
Indeed, $G$ admits a normal form in which each element is expressed uniquely in the form 
\begin{equation}  \label{normal form eq}
a_1^{x_1} \cdots a_k^{x_k} \ c_1^{z_1} \cdots c_{r}^{z_{r}},
\end{equation}
 with $x_i, z_i \in \Z$ for all $i$ with $0 \leq z_{m+j} < o_j$ for $j= 1, \ldots, l$
and $0 \leq x_i < o(a_i)$  if the image of $a_i$ in $A$ has finite order $o(a_i)$.

For all $1 \leq i,j \leq k$, all $1 \leq s \leq m$, and all $1 \leq s' \leq l$, there exist  
unique integers $\gamma_{ijs}$ and $\gamma_{ij (m+s')}$ such that $[a_i, a_j] =  c_1^{\gamma_{ij1}} \cdots c_r^{\gamma_{ijr}}$ in $G$ and $0 \leq \gamma_{ij (m+s')} < o_{s'}$.  Note that $\gamma_{iis}=0$ and  $\gamma_{ijs}=-\gamma_{jis}$.  Also 
 $\gamma_{ii(m+s')}=0$ and, modulo $o_{s'}$, $\gamma_{ij(m+s')}=-\gamma_{ji(m+s')}$. Let $$L \ = \  \max_{ij} \set{ |\gamma_{ij1}|  + \cdots +  |\gamma_{ijr}| \mid 1 \leq i < j \leq k }.$$  

Suppose $n \in \N$ and  consider words $u$ and $v$ on $S$ such that $|u| + |v| \leq n$ and $u \sim v$ in $G$.  We rewrite
 $u$ and $v$ to express them (as elements of $G$)  in normal form
\begin{equation}\label{uv}
u \ = \ a_1^{\alpha_1} \cdots a_k^{\alpha_k} \ c_1^{p_1} \cdots c_r^{p_r} \quad \text{ and } \quad v \ = \ a_1^{\alpha_1} \cdots a_k^{\alpha_k} \ c_1^{q_1} \cdots c_r^{q_r}.
\end{equation}
For all $i$, the exponents of  $a_i$ in the normal  forms of $u$ and of $v$ agree, because the images of $u$ and $v$ in $A$ are equal.  

We will need bounds on the exponents in \eqref{uv} in terms of $n$. Our commutator convention is $[a,b]=a^{-1}b^{-1}ab$.  We have $a_1a_j =  a_ja_1[a_1,a_j]$
and $a_1^{-1}a_j = a_ja_1^{-1}[a_1,a_j]^{-1}$ in $G$  because $[a_1,a_j]$ is central. To obtain the normal forms of $u$ and $v$,  we first push all their occurrences of $a_1^{\pm 1}$  to the left using these identities,  moving  commutators  to the end of the word as they are produced,
using the fact that they are central.    We then push all occurrences of $a_2^{\pm 1}$  to the
left similarly and place them after the power of $a_1$ that we have gathered;  then we push letters $a_3^{\pm 1}$ and so on.  
The total number of times we push one letter $a_i^{\pm 1}$ past another $a_j^{\pm 1}$  to get it in
 the correct position is less than $n^2$, and each push creates one commutator. 
We use the relations $[a_i, a_j] =  c_1^{\gamma_{ij1}} \cdots c_r^{\gamma_{ijr}}$ to rewrite the commutators
that are created
in terms of the letters $c_1, \ldots, c_r$, then, using the fact that they are central,  we shuffle the $c_1, \ldots, c_r$  into order.  It follows
from this description that  $|\alpha_1|+ \cdots + |\alpha_k| \leq n$, and for all $i =1, \ldots, m$, 
\begin{equation} 
|p_i|, |q_i|  \ \leq \  Ln^2,	  \label{pi and qi}
\end{equation}
and, for all $j  =1, \ldots, l$,    
\begin{equation} 
|p_{m+j}|, |q_{m+j}|  \ < \  o_j    \label{pi and qi torsion case}
\end{equation} 
per the normal form \eqref{normal form eq}.

If $w\in G$ conjugates $u$ to $v$, then so does $wc$ for every $c\in\<c_1,\dots,c_m\>$, because the $c_i$ are
central.  Thus $u\sim v$ implies that  there  exist $x_1, \ldots, x_k \in \Z$ so that for  $w =   a_1^{x_1} \cdots a_k^{x_k}$, 
\begin{equation} \label{conj eqn}
w^{-1}uw \ = \ v	
\end{equation}
  in $G$.    By  expressing $w^{-1}uw$ in normal form \eqref{normal form eq} and comparing 
  the powers of each $c_i$ to the powers in the normal form for $v$, we find that \eqref{conj eqn} is equivalent to a system of linear Diophantine equations   $$\begin{array}{rrrrrrlcl}
\l_{1,1} x_1 & + &  \cdots & + &  \l_{1,k} x_k & & &    = & q_1-p_1 \\  
 &   &  &   &   & & &    \vdots &   \\  
\l_{m,1} x_1 & + &  \cdots & + &  \l_{m,k} x_k & & &    = & q_m-p_m \\  
\l_{{m+1}, 1} x_1 & + &  \cdots & + &  \l_{m+1,k} x_k & + & o_1 x_{k+1} &    = & q_{m+1}-p_{m+1} \\  
 &   &  &   &   & & &     \vdots &   \\  
\l_{{r}, 1} x_1 & + &  \cdots & + &  \l_{r,k} x_k &   &   \qquad \ \ + \ \ \  o_l x_{k+l}    & = & q_{r}-p_{r}. 
\end{array}$$
 Here, for all $i$, the $i$-th equation  counts the $c_i$. The additional term $o_j x_{k+j}$ present in equations  $m+1$ through $r$ accounts for $c_{m+j}$ having order $o_j$ --- we need this because we are interested in integer solutions, not solutions  modulo $o_i$.  Recalling that $r = m+l$ and setting $d = k+l$, 
we write this system in the more concise form $M\mathbf{x} = \mathbf{b}$ by defining  the $r \times d$ matrix $M$ and the vectors $\mathbf{x} \in \Z^d$ and $\mathbf{b}  \in \Z^r$ by: 
$$M \ = \  \begin{pmatrix}
 \l_{1,1} & \cdots & \l_{1,k}    \\
 \vdots & & \vdots   \\
 \l_{m,1} & \cdots & \l_{m,k}   \\
 \l_{m+1,1} & \cdots & \l_{m+1,k} &  o_1   \\
 \vdots & & \vdots &  & \ddots & \\
 \l_{m+l,1} & \cdots & \l_{m+l,k}  & & & o_l   \\
\end{pmatrix},
 \quad  
\mathbf{x} \ = \  \begin{pmatrix} 
  x_1 \\ 
   \\   \vdots \\   \\ 
  x_d  \\  \\
\end{pmatrix},
 \quad  \mathbf{b}  \ = \  \begin{pmatrix} 
  q_1 - p_1 \\ 
   \\   \vdots \\   \\ 
  q_r - p_r  \\  \\
\end{pmatrix}.
$$

We claim that  
\begin{equation} \label{kLn}
|\l_{ij}|   \le kLn,
\end{equation}
for all $1\leq i \leq r$ and all $1 \leq j \leq k$.
In order to prove this claim, we consider how to transform $w^{-1}uw$ into normal form,  starting 
from the concatenation of the normal forms for $u$ and $w^{\pm 1}$,  
$$
w^{-1}uw =  (a_k^{-x_k} \cdots a_1^{-x_1}) \ (a_1^{\alpha_1} \cdots a_k^{\alpha_k} \ c_1^{p_1} \cdots c_r^{p_r}) \ 
(a_1^{x_1} \cdots a_k^{x_k}).
$$
First  assume $x_1>0$, so  the first letter of $w$ is $a_1$.  We want to move this letter past $u$
and cancel it with the terminal $a_1^{-1}$ in $w^{-1}$.  Pushing $a_1$ past the central generators $c_1, \ldots, c_r$ has no effect,  while pushing $a_1$ past the syllables $a_1^{\alpha_1}$, \ldots, $a_k^{\alpha_k}$ increases the exponent of $c_i$  (for $1 \leq i \leq r$)  by
$$
\sigma_{i1}= \sum_{t=1}^k \alpha_t\gamma_{1ti}.
$$
If $x_1<0$ then the first letter of $w$ is $a_1^{-1}$ and pushing it past $u$ to cancel with the terminal letter of $w^{-1}$
will decrease the exponent of $c_i$ by $\sigma_{i1}.$ 
Whatever the signs, moving $a_1^{x_1}$ past $u$ to cancel
with the terminal $a_1^{-x_1}$ in $w^{-1}$ will add $\sigma_{i1}x_1$ 
to the exponent of $c_i$.  If we continue in this manner until each syllable $a_1^{x_1}$, \ldots, $a_k^{x_k}$  of $w$ has been moved past $u$
and cancelled with the corresponding $a_1^{-x_1}$, \ldots, $a_k^{-x_k}$ (respectively) in $w^{-1}$, the total change in the exponent of $c_i$ will be 
\begin{equation}\label{mij}
\sum_{j=1}^k  \sigma_{ij}x_j \ \ \text{ where }\ \  
\sigma_{ij}= \sum_{t=1}^k\alpha_{t}\gamma_{jti}.
\end{equation} 
At the end of this process (after shuffling the $c_1, \ldots, c_r$ that have been generated into the correct positions) we have the
normal form for $v$, so the  first   sum in \eqref{mij} equals $q_i-p_i$  modulo the order of $c_i$. 
It follows that for all $1\leq i \leq r$ and all $1 \leq j \leq k$,
$$
\l_{ij} = \sigma_{ij} = \sum_{t=1}^k \alpha_t \gamma_{jti},
$$
which means that  $|\l_{ij}| \le k L \max_t |\alpha_t| \le kLn$, as claimed.

Next we will pursue a change  of variables which will have the effect of replacing the matrix $M$ by a matrix $M'$ which differs in that the entries in the lower-left block --- that is, the $\l_{m+j, t}$ for $j = 1, \ldots, l$ and $t = 1, \ldots, k$ --- are reduced to uniformly bounded values.   To this end, for all such $j$ and $t$, define $s_{m+j, t}$ and $r_{m+j, t}$ to be the integers such that  
\begin{equation}
\l_{m+j, t} \ = \ s_{m+j, t} o_j + r_{m+j, t} \ \text{ and }  \ 0 \leq r_{m+j, t} < o_j.  \label{division}	
\end{equation}

Let $P \in \SL_d(\Z)$ be the lower-triangular matrix 
$$P \ = \  \begin{pmatrix}
 1        \\
  & \ddots &   \\
   &   & 1  \\
 -s_{m+1,1} & \cdots & -s_{m+1,k} & \ 1 \ \\
 \vdots & & \vdots &  & \ \ddots \ & \\
 -s_{m+l,1} & \cdots & -s_{m+l,k} &  & &  \ 1 \  \\
\end{pmatrix}$$
and define
 $$M' \ := \  M P  \ = \  \begin{pmatrix}
 \l_{1,1} & \cdots & \l_{1,k}    \\
 \vdots & & \vdots   \\
 \l_{m,1} & \cdots & \l_{m,k}   \\
 r_{m+1,1} & \cdots & r_{m+1,k} & & o_1    \\
 \vdots & & \vdots & & & \ddots & \\
 r_{m+l,1} & \cdots & r_{m+l,k} & & & & o_l    \\
\end{pmatrix}.$$

We know that there exists $\mathbf{x}$ satisfying $M \mathbf{x}  \ = \ \mathbf{b}$ because $u \sim v$ in $G$. So there exists $\mathbf{x}' =  P^{-1} \mathbf{x} \in \Z^d$  satisfying $M'  \mathbf{x}'   \ = \ \mathbf{b}$.  We will apply a theorem of Borosh, Flahive, Rubin, and Treybig  to the system $M'  \mathbf{x}'   \ = \ \mathbf{b}$.

\begin{thm}[\cite{BFRT-Diophantine}]  \label{BFRT general lemma}
	For all $r >0$ and all $r \times d$ integer matrices $K$ of row-rank $r$ and
	all $\mathbf{b}\in \Z^r$,   if there is a solution $\mathbf{x}\in\Z^d$ to 
	 the system of $r$ linear Diophantine equations $K \mathbf{x}  \ = \ \mathbf{b}$,
then there is a solution $\mathbf{x}\in\Z^d$ whose entries all have absolute value 
at most the maximum of the absolute values of the $r \times r$ minors of the augmented matrix  $\left[ K : \mathbf{b}  \right]$. 
\end{thm}

Assume, as required for Theorem~\ref{BFRT general lemma}, that  $M'$ has full row-rank  $r$.  We bound the absolute values of the $r \times r$  minors of $\left[ M' : \mathbf{b}  \right]$ as follows. Suppose $N$ is an $r \times r$ matrix obtained  by deleting some columns of $\left[ M' : \mathbf{b}  \right]$.  Now consider expanding $\det N$ along the final column of $N$,  which is the only one that could be $\mathbf{b}$.  We get that    $\det N = \sum_{\sigma \in \textup{Sym}(r)} \textup{sign}(\sigma) \pi_{\sigma}$ where each  $\pi_{\sigma}$ is a product of $r$ entries of $N$, one from each of the $r$-rows---and among these $r$ entries, exactly one  comes from the final column of $N$.  The absolute values of the terms from rows $1 \leq i \leq m$ are at most  $kLn$ by \eqref{kLn}, with the possible single exception of one coming from $\mathbf{b}$, which is most $2Ln^2$ by \eqref{pi and qi};  the absolute values of the terms from rows $m +j$ for $1 \leq  j \leq l$ are less  than $o_j$ by \eqref{division}, with the possible exception of one coming from $\mathbf{b}$,  which is less than $2 o_j$ by \eqref{pi and qi torsion case}.  So, for a suitable constant $C>0$,  $$|\det N| \ \leq \ r! \, (kLn)^{m-1} \,  \left( \max \set{kLn, 2Ln^2} \right)  \, 2 o_1 \cdots o_l \ \leq \  C n^{m+1},$$  
 and Theorem~\ref{BFRT general lemma} tells us that the system $M'  \mathbf{x}'   \ = \ \mathbf{b}$ has a solution  $\mathbf{x}'$ whose entries $x'_i$  all satisfy  $|x'_i|  \leq C n^{m+1}$.

For $1 \leq i \leq k$, the entry $x_i$ of $\mathbf{x} = P \mathbf{x}'$ is $x'_i$.  
 Thus we obtain a word $w=a_1^{x_1}\cdots a_k^{x_k}$ of length at most a constant times $n^{m+1}$  such that $uw=wv$ in $G$.

Finally, suppose that the row-rank of $M'$ is   $\hat{r} < r$. Then some row of $M'$ is a $\Q$-linear combination of the other rows, and so, because $M'  \mathbf{x}'   \ = \ \mathbf{b}$  is consistent, the same row in     $\left[ M' : \mathbf{b}  \right]$ is the same $\Q$-linear combination of the other rows.  So removing this row does not alter the set of solutions. 
We discard rows in this manner until we have replaced $M'$ with a matrix of full row-rank $\hat{r}$.  Theorem~\ref{BFRT general lemma} 
then tells us that there is a solution $\mathbf{x}'$  with  $|x'_1|$, \ldots, $|x'_d|$ all at most the maximum of the absolute values of the $\hat{r} \times \hat{r}$ minors for the redacted matrix,  
which then leads to a stronger bound than the one  we derived above. 

This completes our proof of Theorem~\ref{thm: class-2 upper bound}. \qed

\section{Conjugator length in the groups $G_m$}

The group $G_m$ of Theorem~\ref{central extensions CL lower bound examples}
is a central extension of $\Z^{m+2} = \langle a_1, \ldots, a_m, b_1, b_2\rangle$ by $\Z^m  = \langle c_1, \ldots, c_m \rangle$.
Elements $w$ of $G_m$ can be expressed uniquely per the normal form 
\begin{equation} \label{normal form Gm}	
w \ = \  a_1^{x_1} \cdots a_m^{x_m} \  b_1^{y_1} b_2^{y_2} \ c_1^{z_1} \cdots c_m^{z_m}
\end{equation}
with $x_1, \ldots, x_m, y_1, y_2, c_1, \ldots, c_m \in \Z$.  

\begin{proof}[Proof of Theorem~\ref{central extensions CL lower bound examples}]
Theorem~\ref{thm: class-2 upper bound} tells that $\CL(n) \preceq n^{m+1}$, so 
what remains to be proved is that  $\CL(n) \succeq n^{m+1}$.  Suppose $n \in \N$.  Let $u = b_1 b_2^n \ a_1^{-n} b_1^{-n}      a_1^n b_1^n  $ and $v=b_1 b_2^n$, which we will see momentarily are conjugate elements of $G_m$.   Then 
\begin{equation} \label{sum of lengths}	
|u| + |v|  \ = \  2+6n
\end{equation}
 and $u = b_1 b_2^n  c_1^{-n^2}$ in $G_m$. 
For $w$ as in  \eqref{normal form Gm},   we use the defining relations for $G_m$ to calculate the normal forms of $uw$ and $wv$, pushing the  letters $a_i$ then $b_i$ to the left and remembering that the $c_i$ are central: 
\begin{align*}
uw  & \  = \   b_1 b_2^n  c_1^{-n^2} \  a_1^{x_1} \cdots a_m^{x_m} \  b_1^{y_1} b_2^{y_2} \ c_1^{z_1} \cdots c_m^{z_m}  \\  
& \ = \ a_1^{x_1} \cdots a_m^{x_m} \  b_1^{y_1 +1} b_2^{y_2+n} \ c_1^{x_1  + z_1 - n^2} c_2^{x_2 -nx_1  + z_2} \cdots   c_m^{x_m  - n x_{m-1} + z_m}, \\[12pt]
wv  & \  = \      a_1^{x_1} \cdots a_m^{x_m} \  b_1^{y_1} b_2^{y_2} \ c_1^{z_1} \cdots c_m^{z_m} \ b_1 b_2^n     \\  
& \ = \  a_1^{x_1} \cdots a_m^{x_m} \  b_1^{y_1 +1} b_2^{y_2 +n} \ c_1^{z_1} \cdots c_m^{z_m}.   
\end{align*}

Now,  $uw  =  wv$  in $G_m$ if and only if the exponents of $c_1$, \ldots, $c_m$ in their normal forms agree, and that amounts to the system of equations
$$\begin{array}{rrrrrcrrrll}
 &  x_1 & \!\!   \!\! &  &  &  & & &  & & = \ n^2  \\   	
 - & nx_1  & \!\! + \!\! & x_2 & & & & & & & = \ 0   \\ 	
& & & \rotatebox{20}{$\ddots$}   & & \rotatebox{20}{$\ddots$}  & & & & &  \,\, \vdots \\
&  &  &   & - &  nx_{m-2} & + &  x_{m-1} & & &  = \ 0   \\ 	
 & & & & & & - &  n x_{m-1}  & \!\! + \!\! & x_m & = \ 0   	
\end{array}$$
with no constraints on $y_1, y_2, z_1, \ldots, z_m$.  This system has the unique solution $x_i = n^{i+1}$    
for all $i$. 
So  
\begin{equation}
	w_0  \ = \ a_1^{n^2} a_2^{n^3} \cdots a_m^{n^{m+1}} 
\end{equation}
satisfies $uw_0=w_0v$ in $G_m$ and therefore $u$ and $v$ are conjugate in $G_m$.   Moreover, \emph{any} word $w$  such that $uw=wv$ in $G_m$ 
has the same image as $w_0$ under the retraction $G_m\onto \<a_1,\dots,a_m\>\cong\Z^m$ (which
has kernel $\<b_1,b_2,c_1,\dots,c_m\>$), so $w$ must contain   at least $n^{m+1}$ occurrences
of the letter $a_m$.   
Therefore $\CL(u,v) \geq n^{m+1}$ and, in light of \eqref{sum of lengths},  $\CL(n) \succeq n^{m+1}$.
\end{proof}

\bibliographystyle{alpha}
\bibliography{bibli}

\ni {Martin R.\ Bridson},  Mathematical Institute, Andrew Wiles Building, Oxford OX2 6GG, United Kingdom. {bridson@maths.ox.ac.uk}, \
\href{http://www2.maths.ox.ac.uk/~bridson/}{https://people.maths.ox.ac.uk/bridson/}

\ni  {Timothy R.\ Riley}, \rule{0mm}{6mm} 
Department of Mathematics, 310 Malott Hall,  Cornell University, Ithaca, NY 14853, USA. {tim.riley@math.cornell.edu}, \
\href{http://www.math.cornell.edu/~riley/}{http://www.math.cornell.edu/$\sim$riley/}

 \end{document}